\documentclass[12pt, reqno,toc=indent]{amsart}

\usepackage[margin=1.5in]{geometry}

\usepackage[utf8]{inputenc}
\usepackage[T1]{fontenc}
\usepackage[english]{babel}

\usepackage{graphicx}
\usepackage{tikz-cd}
\usepackage{import}
\usepackage{wrapfig}
\usepackage{mdframed}

\usepackage{tabularx}
\usepackage{enumitem}
\usepackage{amssymb}
\usepackage{amsmath}
\usepackage{amsfonts}
\usepackage{amsthm}
\usepackage{mathrsfs}

\usepackage[breakable,skins]{tcolorbox}

\newtcolorbox{activitybox}[1][]{%
    breakable,
    enhanced,
    colback=gray!20,
    colframe=gray!30,
    coltitle=black,
    #1
}

\usepackage{amsmath, amsthm, amssymb, amsfonts}

\usepackage{xcolor}
\definecolor{antiquefuchsia}{rgb}{0.57, 0.36, 0.51}
\definecolor{azure}{rgb}{0.0, 0.5, 1.0}

\usepackage[colorlinks = true, linkcolor = azure, citecolor = antiquefuchsia]{hyperref}

\usepackage{comment}
\usepackage{todonotes}

\theoremstyle{plain}
\newtheorem{theorem}{Theorem}
\newtheorem{proposition}[theorem]{Proposition}

\newtheorem{lemma}[theorem]{Lemma}

\theoremstyle{definition}

\theoremstyle{remark}
\newtheorem{remark}{Remark}

\newcommand{\R}{\mathbb{R}} 
\newcommand{\C}{\mathbb{C}}
\newcommand{\DD}{\mathbb{D}} 
\newcommand{\D}{\mathscr{D}}

\newcommand{\te}{\tilde{\tau}_{\lambda}}

\newcommand{\Alpha}{\boldsymbol{\alpha}}

\DeclareMathOperator{\Diff}{Diff}

\newcommand{\Id}{\mathrm{Id}}
\renewcommand{\Re}{\operatorname{Re}}
\renewcommand{\Im}{\operatorname{Im}}

\title{Biholomorphism rigidity for transport twistor spaces}

\author[J. Bohr]{Jan Bohr}
\address{ 
Mathematical Institute of the University of Bonn,
Endenicher Allee 60, 53115 Bonn, Germany}
\email {bohr@math.uni-bonn.de}

\author[F. Monard]{Fran\c{c}ois Monard}
\address{Department of Mathematics, University of California, Santa Cruz, CA 95064, USA}
\email{fmonard@ucsc.edu}

\author[G.P. Paternain]{Gabriel P. Paternain}
\address{ Department of Mathematics, University of Washington, Seattle, WA 98195, USA}
\email {gpp24@uw.edu}

\date{\today}

\begin{document}

\maketitle

\begin{abstract} We prove that biholomorphisms between the transport twistor spaces of simple or Anosov surfaces exhibit rigidity: they must be, up to constant rescaling and the antipodal map, the lift of an orientation preserving isometry.
\end{abstract}

\section{Introduction}
Transport twistor spaces are degenerate complex 2-dimensional manifolds, which can be associated with any oriented Riemannian surface \((M,g)\). The complex geometry of these spaces is intricately linked to the geodesic flow of the surface.

For an oriented Riemannian surface \((M,g)\), possibly with a non-empty boundary $\partial M$, the transport twistor space is defined as the 4-manifold
\[
Z = \{(x,v) \in TM : g(v,v) \leq 1\},
\]
endowed with a natural complex structure. This structure turns the interior \( Z^\circ \) into a classical complex surface, though it degenerates at the unit circle bundle \( SM \subset \partial Z \), encoding the geodesic vector field in the process.

In recent years, these twistor spaces have emerged as valuable tools for organising and reinterpreting various questions in geometric inverse problems and dynamical systems \cite{BoPa23, BLP23, BMP24}.  

This note examines biholomorphisms between such transport twistor spaces, establishing rigidity results in two distinct contexts: \textit{simple} surfaces and \textit{Ano\-sov} surfaces. A simple surface is a compact Riemannian surface \((M,g)\) with a strictly convex boundary, no trapped geodesics, and no conjugate points. An Anosov surface is a closed, oriented Riemannian surface \((M,g)\) whose geodesic flow is Anosov—equivalently, \( g \) belongs to the \( C^2 \)-interior of the set of metrics without conjugate points.

Consider two oriented Riemannian surfaces \((M_i, g_i)\) (for \( i=1,2 \)) and their associated twistor spaces \( Z_1 \) and \( Z_2 \). A map \(\Phi\colon Z_1 \rightarrow Z_2\) is called a biholomorphism if it is a diffeomorphism that restricts to a biholomorphism between the complex surfaces \( Z^\circ_1 \) and \( Z^\circ_2 \). 
If two twistor spaces are biholomorphic, the underlying real surfaces must be diffeomorphic. Therefore, without loss of generality, we may fix the surface \( M \) and consider different Riemannian metrics on it. Biholomorphisms from $Z$ to itself are called {\it automorphisms} and the group of automorphisms is denoted $\mathrm{Aut}(Z)$.

Obvious examples of biholomorphisms arise from the lifting of isometries: if \(\varphi\colon (M, g_1) \to (M, g_2)\) is an orientation preserving isometry, then the \textit{lift}
\[
\varphi_\sharp\colon Z_1 \to Z_2, \quad \varphi_\sharp(x, v) = (\varphi(x), d\varphi_x(v)),
\]
is a biholomorphism. Moreover, the geodesic flow being reversible implies that the antipodal map \( a(x, v) = (x, -v) \) is an automorphism. Constant rescalings also produce biholomorphisms: if \( g_2 = Cg_1 \) for a constant \( C > 0 \), then also the map $\mathrm{sc}(x, v) = (x, C^{-1/2}v)$ is a biholomorphism $\mathrm{sc}\colon Z_1\to Z_2$.


A simple yet illustrative example of a biholomorphism that is {\it not} of the aforementioned types arises for \( M = \mathbb{R}^2 \) with its standard Euclidean metric. Here, \( Z \) can be identified with \( \mathbb{C} \times \mathbb{D} \), where \( \mathbb{D} \) is the closed unit disc. Using standard coordinates \((z, \mu) \in \mathbb{C} \times \mathbb{D}\), the holomorphic structure is described in terms of \( T^{0,1}Z =\C(
\partial_{\bar{z}} + \mu^2 \partial_z )
\oplus \C  \partial_{\bar{\mu}}.
$
One can verify that the map
\[
(z, \mu) \mapsto \left(z + \frac{is\mu}{1 + |\mu|^2}, \mu\right)
\]
is an automorphism  of $Z$ for any \( s \in \mathbb{R} \). However, for $s\neq 0$ this cannot be expressed in terms of the antipodal map and lifts of isometries.
 Note that this automorphism commutes with the lifts of translations, allowing it to descend to an automorphism of the transport twistor space of the 2-torus.
 
  Our primary results demonstrate that biholomorphism rigidity holds for both simple and Anosov surfaces:

\begin{theorem}\label{mainthm}
Let \( g_1 \) and \( g_2 \) be two simple metrics on $M$ with \( \mathrm{Vol}(\partial M, g_1) = \mathrm{Vol}(\partial M, g_2) \). If \( \Phi\colon  Z_1 \to Z_2 \) is a biholomorphism, then there exists an orientation preserving isometry \( \varphi\colon  (M, g_1) \to (M, g_2) \) such that:
\[
\Phi = \varphi_\sharp \quad \text{or} \quad \Phi = \varphi_\sharp \circ a.
\]
\end{theorem}

\begin{theorem}\label{mainthm_anosov}
Let \( g_1 \) and \( g_2 \) be two Anosov metrics on \( M \) with \( \mathrm{Vol}(M, g_1) = \mathrm{Vol}(M, g_2) \). If \( \Phi\colon Z_1 \to Z_2 \) is a biholomorphism, then there exists an orientation preserving isometry \( \varphi\colon  (M, g_1) \to (M, g_2) \) such that
\[
\Phi = \varphi_\sharp \quad \text{or} \quad \Phi = \varphi_\sharp \circ a.
\]
\end{theorem}

\begin{remark}
If \( C = \mathrm{Vol}(\partial M, g_1) / \mathrm{Vol}(\partial M, g_2) \neq 1 \), then Theorem \ref{mainthm} still holds with \( \varphi \) being an isometry between \( g_1 \) and \( Cg_2 \), and \( \Phi(x, v) = (\varphi(x), C^{1/2}d\varphi_x(v)) \), possibly composed with the antipodal map. Theorem \ref{mainthm_anosov} can be modified analogously, with $C$ being a ratio of areas.
\end{remark}


The proofs of these theorems rely on several key ingredients:
\begin{itemize}[itemsep=3pt]
    \item
    Ideas from the proof of boundary rigidity for simple surfaces by Pestov and Uhlmann \cite{PeUh05}, and the recent proof of marked length spectrum rigidity for Anosov surfaces \cite{GLP23}.
    \item
    A characterisation of biholomorphisms \( \Phi\colon Z_1 \to Z_2 \) as orientation preserving diffeomorphisms that are {\it fibrewise holomorphic} and induce an orbit equivalence \( \phi=\Phi|_{SM_1}\colon SM_1 \to SM_2 \)  between geodesic flows.
    \item
    The identity principle for holomorphic maps, as established in \cite[Corollary 1.7]{BMP24}, which asserts that two holomorphic maps \( \Phi, \Psi\colon Z_1 \to Z_2 \) such that $\Phi|_{SM_{1}}=\Psi|_{SM_{1}}$ must agree everywhere.
\end{itemize}

 \subsection{Relationship to inverse problems} The question of biholomorphism rigidity naturally comes up in geometric inverse problems. Given metrics $g_1$ and $g_2$ on $M$, recall that a diffeomorphism $\phi\colon SM_1\to SM_2$ is called
\begin{itemize}
	\item[$\circ$] a {\it conjugacy}, if it intertwines the geodesic flows; 
	\item[$\circ$] an {\it orbit equivalence}, if it intertwines geodesic flows up to a time change.
\end{itemize}
Let  now $(M,g_0)$ be a simple surface and consider
\begin{eqnarray*}
\mathbb{M}_\mathrm{s} = \{g: \text{simple metric such that } g= g_0 \text{ on } TM|_{\partial M}\}
\end{eqnarray*}
 Each $g\in \mathbb{M}_\mathrm{s}$ induces a diffeomorphism $\alpha_g\in \mathrm{Diff}(\partial SM)$, called the {\it scattering relation} (cf.~Section \ref{section:preliminaries}).
Any two metrics $g_1,g_2\in \mathbb{M}_\mathrm{s}$ are related by an orbit equivalence\footnote{One constructs $\phi$ by flowing back along the first geodesic flow and forward along the second, achieving smoothness by moving slightly across $\partial SM$ --- see e.g.~\cite[Section 5.3]{BMP24}.}  $\phi$ 
 and any such orbit equivalence satisfies
\begin{equation*}
	\alpha_{g_2}\circ(\phi|_{\partial SM}) = (\phi|_{\partial SM})\circ \alpha_{g_1}\quad \text{ on } \partial SM.
\end{equation*}
Imposing the boundary condition $\phi|_{\partial SM} =\mathrm{Id}$ implies that $\alpha_{g_1}=\alpha_{g_2}$ and vice versa, if the scattering relations agree, then the metrics are related by a boundary fixing conjugacy --- as a consequence of the {\it scattering rigidity} proved in \cite{PeUh05}, this enforces that $g_1$ and $g_2$ are isometric via a boundary fixing isometry. Our first theorem demonstrates that this rigidity phenomenon persists, if the boundary condition is replaced by the requirement that $\phi$ holomorphically extends to transport twistor space. 

While our result relies on ideas from \cite{PeUh05}, the knowledge that no information is lost by focusing on transport twistor spaces suggests that {\it conjugacy rigidity problems} (cf.\,\cite[Section 4.6]{GuMa24}) might be amenable to using complex geometric methods on twistor space. However, such an approach requires a better understanding of which  orbit equivalences admit holomorphic extensions.

Besides proving rigidity, it is of interest to understand the structure of  the range $\mathbb A = \{\alpha_g:g\in \mathbb M_{\mathrm{s}}\}\subset \mathrm{Diff}(\partial SM)$ of the scattering map. In analogy to the range characterisations in \cite{Sha11, PeUh04,BoPa23}, one might hope that $\mathbb A$ is of the form
\[
	\mathbb A_{\mathcal H}=\{\psi^{-1}\circ \alpha_{g_0}\circ \psi: \psi\in \mathcal H\},\quad \mathcal H\subset \mathrm{Diff}(\partial SM),
\]	
where $\mathcal H$ is a suitable class of diffeomorphisms. Note that if $\mathcal H =\mathrm{Diff}(\partial SM)$, the aforementioned intertwining relation implies that $\mathbb A\subset \mathbb A_{\mathcal H}$, however the latter set also contains scattering data of non-simple metrics.
Guided by the just cited articles, where it was key to respect the complex structure of $M$ and its transport twistor space, one is tempted to consider
\[
	\mathcal H =\{\Phi|_{\partial SM}: \Phi\colon Z_{g}\to Z_{g_0} \text{ biholomorphism},~ g\in \mathbb M_\mathrm{s}\}.
\]
A consequence of Theorem \ref{mainthm} is that such a guess fails:  the orbit $\mathbb{A}_{\mathcal{H}}$ then only contains scattering data of metrics that are isometric to $g_0$. In particular, new ideas are needed in order to obtain a range characterisation for the scattering relation akin to those found in the Calder{\'o}n- and (linear and non-linear) X-ray tomography problems. 

\subsection*{Organisation of the article}
Section \ref{section:preliminaries} provides preliminaries on transport twistor spaces and properties of biholomorphisms. Section \ref{section:general} offers three general rigidity results, one for diffeomorphisms preserving the canonical contact 1-form, one for geodesic equivalences and another one for time changes. The proof of Theorem \ref{mainthm} is presented in Section \ref{section:mainthm}, and the proof of Theorem \ref{mainthm_anosov} is given in Section \ref{section:mainthm_anosov}.

\subsection*{Acknowledgements}   FM was partially supported by NSF-CAREER grant DMS-1943580 and GPP was partially supported by NSF grant DMS-2347868.

\section{Preliminaries}\label{section:preliminaries}

Let $(M,g)$ be a connected and oriented surface, possibly with a non-empty boundary $\partial M$. We also view $(M,g)$ as Riemann surface and write
\[
	\mu \cdot v = (\Re \mu) v + (\Im \mu) v^\perp,\quad v\in T_xM, \mu\in \C
\]
for the associated complex multiplication on $T_xM$, with $v^\perp$ being the rotation by $\pi/2$, counterclockwise according to the orientation.
If $\partial M\neq 0$, we write $\nu$ for the inward pointing unit normal to $\partial M$ and define $\nu_\perp =-\nu^\perp$.


 The unit tangent bundle $SM$ is a $3$-manifold,  possibly with a non-empty boundary $\partial SM$. The {\it geodesic flow} on $SM$ is denoted with $(\varphi_t)$.  The tangent bundle $TSM$ comes with a natural frame \[\{X,H,V\},\] where $X$ generates the  geodesic flow, $V$ generates the vertical flow $(x,v)\mapsto (x,e^{it} v)$, and $H=[V,X]$ is defined as commutator. The {\it Sasaki metric} on $SM$ is defined by requiring this frame to be orthonormal.

If $\partial M\neq \emptyset$ is strictly convex, the {\it scattering relation} of $g$ is the set
 \[
 	\Big\{\big((x,v),(y,w)\big)\in \partial SM \times \partial SM: \begin{array}{l} (y,w)=\varphi_t(x,v) \text{ for some } t\neq 0 \\
 	\text{or } x=y \text { and }  v=w\in T_x(\partial M)
 	\end{array} \Big\}.
 \] 
If $g$ is simple, this is the graph of a smooth diffeomorphism $\alpha\in \mathrm{Diff}(\partial SM)$ that we also refer to as scattering relation. We have $\alpha^2 =\mathrm{Id}$ and further $\alpha$ has the {\it glancing region}   $\partial_0SM = S( \partial M)$ as fixed point set.

\subsection{Transport twistor space} We equip the total space of the unit disk bundle $Z\to M$ with an involutive complex $2$-plane bundle $\mathscr D\subset T_\C Z$ (also referred to as {\it involutive structure} \cite{Tre92}), having the following properties:
\begin{enumerate}[label=\rm(\alph*)]
	\item 
	$
		\D \cap \bar \D = \begin{cases} 
			\C X & SM\\
			0 & Z\backslash SM
		\end{cases}		
	$\\[.3em]
	Away from $SM$ this implies that  $\mathscr D = \ker(J+i)$ for a complex structure $J$. In particular, $(Z^\circ,J)$ is a classical complex surface.\\[-1em]
	\item Let $\pi\colon Z\to M$ be the projection map. Equipping the fibres $Z_x=\pi^{-1}(\{x\})$ with their standard complex structure (induced by $g$ and the orientation), the embedding $Z_x\hookrightarrow Z$ is holomorphic. 
	\\[-1em]
	\item The orientation induced by $J$ is the canonical one on $TM$.
\end{enumerate}
For the construction of $\D$ see \cite{BoPa23}; in fact, $\D$ is uniquely characterised by these three properties. For $M=\R^2$ and $Z= \{(z,\mu)\in \C^2:|\mu|\le 1\}$, the interested reader may check at once that $\D=\C(\partial_{\bar z} + \mu^2 \partial_{z})\oplus \C \partial_{\bar \mu}$.

\begin{remark}\label{rem_proj}
The construction of transport twistor space is inspired by the more classical {\it projective twistor space} $Z_\mathbb P$, which has been used in the context of Zoll structures \cite{LeMa02, LM10,Roc11} and projective structures \cite{Met21,MePa20}. 
\end{remark}

\subsubsection{Holomorphic functions}
 For $U\subset Z$ open, we denote with
\[
	\mathcal A(U)=\{f\in C^\infty(U): df(\D)=0\}
\]
the algebra of holomorphic functions on $U$, understood to be smooth up to the boundary, if $U\cap \partial Z\neq \emptyset$. The restriction $f\mapsto f|_{SM}$ is an isomorphism
\[
	\mathcal A(Z) \xrightarrow{\sim} \left\{u\in C^\infty(SM): Xu = 0 \text{ and } u \text{ is fibrewise holomorphic} \right\},
\]
where $u$ being {\it fibrewise holomorphic} means that for all $x\in M$ the function $u(x,\cdot)\colon S_x M   = \partial Z_x \to \C$ extends to a holomorphic map on the fibre $Z_x$, or what is equivalent, all negative Fourier modes vanish (cf.~\cite[Proposition 4.4]{BoPa23}).

For simple surfaces more can be said.
\begin{proposition}\label{prop_scathol} If $(M,g)$ is simple, then $f\mapsto f|_{\partial SM}$ is an isomorphism
\[
	\mathcal A(Z) \xrightarrow{\sim} \left\{u\in C^\infty(\partial SM): u=u\circ \alpha \text{ and } u \text{ is fibrewise holomorphic} \right\}.
\]
\end{proposition}

\begin{proof}
By \cite[Theorem 5.1.1]{PSU23} every $\alpha$-invariant function $u\in C^\infty(\partial SM)$ extends to a smooth solution of $Xu = 0$ on $SM$ and the proof of \cite[Lemma 11.5.3]{PSU23} shows that $u|_{\partial SM}$ being fibrewise holomorphic implies that $u$ is fibrewise holomorphic on all of $SM$. By the discussion above, this implies that the restriction map $f\mapsto f|_{\partial SM}$ is onto. Injectivity is obvious. 
\end{proof}

Equipping $M$ with the complex structure induced by $g$ and the orientation, the zero section embedding $M\hookrightarrow Z$ is holomorphic, and hence there is a restriction map as follows:
\[
	\mathcal A(Z)\to \mathcal A(M),\quad f\mapsto f|_M
\]
Here $\mathcal A(M)$ is  the algebra of holomorphic functions on $M$, smooth up to the boundary. A classical result of Pestov--Uhlmann \cite{PeUh05} can then be rephrased as a type of Cartan extension theorem:

\begin{proposition}[Pestov--Uhlmann extension -- Corollary 4.7 in \cite{BoPa23}] If $(M,g)$ is simple, then the restriction map $\mathcal A(Z)\to \mathcal A (M)$ is onto.\qed
\label{puext1}
\end{proposition}

More generally,  for fixed $x\in M$ there is a Taylor expansion $f(x,v) =f(x,0)+ \theta_x(v) + O(|v|^2)$ about $v=0$ and holomorphicity of $f(x,\cdot)\colon Z_x\to \C$ implies that $\theta$ is a $(1,0)$-form on $M$. In fact, $\theta \in \mathcal H_1(M)$, the space of {\it holomorphic  $1$-forms}, again understood to be smooth up to the boundary.

\begin{proposition}[Pestov--Uhlmann extension for holomorphic $1$-forms] If $(M,g)$ is simple, then the map $\mathcal A(Z)\to \mathcal H_1(M)$, $f\mapsto \theta$ is onto. (And the preimage of $\theta$ may be chosen such that $f\circ a = -f$).
\label{puext2}
\end{proposition}

\begin{proof}[Proof (using the notation from Section 6.1 in \cite{PSU23})]  By \cite[Theorem 12.2.4]{PSU23}, gi\-ven any $g\in \Omega_m$ we can find a smooth solution $u$ to $Xu=0$ whose
$m$th Fourier mode $u_m$ is given by $g$.  Giving an element $\theta\in \mathcal H_1(M)$ is equivalent to giving $g\in \Omega_{1}$ with $\eta_{-}g=0$. Thus $v:=\sum_{k\geq 0}u_{2k+1}$ is fibrewise holomorphic such that $Xv=0$ and with $v_1=g$. Its unique extension $f$ to $Z$ is the required function.
\end{proof}

For closed Anosov surfaces, where $\mathcal A(Z)\cong \mathcal A(M)\cong \C$, there are suitable replacements  \cite[Theorem 1.2]{BLP23} -- here, we will not need this directly, however.


\subsubsection{Holomorphic maps}
 A {\it holomorphic map} $\Phi\colon (Z,\D)\to (Z',\D')$ into another transport twistor space (or a into classical complex manifold with $\D' = T^{0,1}Z'$) is {\it per definitionem} a map that is smooth up to the boundary of $Z$ and that satisfies
\[
	d\Phi_{(x,v)}(\D_{(x,v)}) \subset \D'_{\Phi(x,v)}\quad\text {for all } (x,v)\in Z.
\]
As this is a closed condition, it suffices to verify holomorphicity in the interior $Z^\circ$, where it corresponds to the familiar notion from complex geometry.

\begin{proposition}\label{prop_hmaps}
	Let $\Phi\colon Z\to Z'$ be an orientation preserving diffeomorphism between transport twistor spaces. If $\Phi$ is a biholomorphism, then:
	\begin{enumerate}[label=\rm(\roman*)]
		\item\label{hmaps1} $\Phi$ restricts to a diffeomorphism  $\phi \colon SM\to SM'$ such that 
		\[
			\phi_*X \in \R X'\quad \text{ and } \quad \phi_*V\in \R X' \oplus \R V'.
		\]
		\item\label{hmaps2} the map $\Phi(x,\cdot)\colon Z_x\to Z'$ is holomorphic for all $x\in M$.
	\end{enumerate}
\end{proposition}

\begin{remark}
	 Using that $\D$ is {\it uniquely characterised} by the above three properties, one can show that the conditions (i) and (ii) are also {\it sufficient} for $\Phi$ to be a biholomorphism, but we will not need this here. 
\end{remark}


\begin{proof}[Proof] Consider $(x,v)\in SM$ and
\[
	A=d\Phi_{(x,v)}\colon (T_\C Z)_{(x,v)} \to (T_\C Z')_{\Phi(x,v)}.
\]
This maps $\D(x,v)$ into $\D'(\Phi(x,v))$ and, being the complexification of a real linear map, satisfies $A(\bar w) = \overline{Aw}$. 
Thus $X(x,v)$ is sent to a nonzero  real vector inside $\D'(\Phi(x,v))$, which enforces
\[
	\Phi(x,v)\in SM'\quad \text{ and } \quad \Phi_*X(x,v) \in \R X'(\Phi(x,v)).
\]
Hence $\phi=\Phi|_{SM}$ is a smooth map $SM\to SM'$ and satisfies the first condition in \ref{hmaps1}; repeating the argument with $\Phi^{-1}$ shows that $\phi$ is a diffeomorphism. 

Next, we claim that \begin{equation}\label{dplusdbar}
	(\D\oplus\bar \D)\cap TSM  = \R X\oplus \R V.
\end{equation}
Indeed, the inclusion $Z_x\hookrightarrow Z$ being holomorphic implies that $TZ_x=T^{0,1}Z_x\oplus T^{1,0}Z_x \subset \D\oplus \bar \D$ and thus on $SM$ both $V$ and $V_\perp$ (the normal vector to $SM\subset TM$) lie in $\D\oplus \bar \D$. Thus, on $SM$ we have $\C X\oplus \C V\oplus \C V_\perp \subset \D\oplus \bar \D$; for dimension reasons we must have equality, and intersecting with the real tangent bundle $TSM$ yields \eqref{dplusdbar}.

Using  \eqref{dplusdbar} on both $SM$ and $SM'$ we conclude that $A$ sends $V(x,v)$ into $\R X' \oplus \R V'$, which the second half of \ref{hmaps1}. Property \ref{hmaps2} is obvious.
\end{proof}

In \cite{BMP24} we showed that $\D$ is {\it locally integrable}, also near points on $SM$, where the Newlander--Nirenberg theorem cannot be applied due to the degeneracy of $\D$. This has the following consequence:

\begin{proposition}[Corollary 1.7 in \cite{BMP24}]\label{prop_identityprinciple} Let $\Phi_1,\Phi_2\colon Z\to Z'$ be two holomorphic maps with $\Phi_1=\Phi_2$ on $SM$. Then $\Phi_1=\Phi_2$ everywhere.
\end{proposition}

\subsection{Rad{\'o}-Kneser-Choquet theorem}\label{sec_RKC}

Consider the {\it Hardy space} \[\mathbb H= \{h\in L^2(\mathbb S^1): h_k=0, k<0\}\] of $L^2$-functions on $\mathbb S^1$ with vanishing negative Fourier modes. Any function $\psi\in C^\infty(\mathbb S^1)\cap \mathbb H$ has a unique holomorphic extension $\Psi\in C^\infty(\DD)$ and the Rad{\'o}--Kneser--Choquet theorem (see e.g.~\cite[Chapter 3.1]{Dur04}) asserts that
\begin{equation*}\label{RKC}
\psi \in \mathrm{Diff}(\mathbb S^1)\quad \Rightarrow \quad \Psi\in \mathrm{Diff}(\DD).
\end{equation*}
Here $\mathrm{Diff}(\cdot)$ denotes the diffeomorphisms of a manifold. From this we deduce a rigidity result:
\begin{lemma}\label{circlediff} Suppose $\psi \in \Diff(\mathbb S^1)$ satisfies the following properties:
\begin{enumerate}[label=\rm{(\roman*)}]
	\item\label{circlediff1} $\psi(1)=1$;
	\item\label{circlediff2} $\psi(-\mu)=-\psi(\mu)$ for all $\mu\in \mathbb S^1$;
	\item\label{circlediff3} $\psi^*(\mathbb H)\subset \mathbb H$.
\end{enumerate}
Then $\psi=\mathrm{Id}$.
\end{lemma}

\begin{proof}
	Consider $\psi$ as function $\psi\colon \mathbb S^1\rightarrow \C$, then applying \ref{circlediff3} to $f(\mu)=\mu$ gives that $\psi\in \mathbb H$. Consequently, there is a holomorphic map $\Psi\colon \DD\rightarrow \C$ with $\Psi =\psi$ on $\mathbb S^1$. By the Rad{\'o}--Kneser--Choquet theorem, $\Psi\colon \DD\rightarrow \DD$ is a biholomorphism and thus there are  $(a,u)\in \DD^\circ\times \mathbb S^1$ such that  $\Psi(\mu)=u(\mu-a)/(1-\bar a\mu)$.
 Property  \ref{circlediff2} can be analytically continued to $\Psi(-\mu)=-\Psi(\mu)$ for  all $\mu \in \DD$ and thus $a=-\bar u \Psi(0)=0$. Further, $u=\Psi(1)=1$, hence $\psi=\mathrm{Id}$.
\end{proof}
The Rad{\'o}--Kneser--Choquet theorem can also be leveraged to  Riemannian surfaces $(M,g)$. Denote $\mathcal A(\partial M,g)\subset C^\infty(\partial M)$ the space of boundary values of $g$-holomorphic functions.

\begin{lemma}\label{rkcriemann}
	Suppose $M$ is diffeomorphic to a disk and $g_1$ and $g_2$ are two Riemannian metrics. Then any $\varphi\in \Diff(\partial M)$ with $\varphi^*\mathcal A(\partial M,g_2)\subset \mathcal A(\partial M,g_1)$ can be extended to a biholomorphism $\varphi\colon (M,g_1)\to (M,g_2)$.
\end{lemma}

\begin{proof} By the Riemann mapping theorem, for $k=1,2$, there is a biholomorphism $\chi_k\colon (M,g_k)\to \DD$ which restricts to a diffeomorphism $\partial M \to \mathbb S^1$ satisfying $(\chi_k|_{\partial M})^*(\mathbb H\cap C^\infty(\mathbb S^1))= \mathcal A(\partial M,g_k)$. Thus the composition  $\psi =(\chi_2|_{\partial M})\circ \varphi\circ (\chi_1|_{\partial M})^{-1}$ belongs to $\Diff(\mathbb S^1)$ and preserves $\mathbb H\cap C^\infty(\mathbb S^1)$. Viewed as a $\C$-valued map, this ensures that $\psi\in \mathbb H$ and by the  Rad{\'o}--Kneser--Choquet theorem, $\psi$ extends to an automorphism $\Psi\in \mathrm{Aut}(\DD)$. The sought after extension of $\varphi$ is then given by $\chi_2^{-1} \circ \Psi\circ \chi_1\in \Diff(M)$.
\end{proof}

\subsection{A boundary determination lemma}\label{seclambda}In this subsection we assume that $\partial M\neq \emptyset$ and consider flows on $SM$ generated by the vector fields \[
F = X + \lambda V,\quad \lambda\in C^\infty(SM,\R).
\]
We assume that the flow $(\phi_t)=(\phi^\lambda_t)$ of $F$ is non-trapping, such that the scattering relation $\alpha_\lambda \colon \partial SM\to \partial SM$ is well-defined. We also assume that $\partial M$ is strictly $F$-convex, meaning that 
\[\Pi_{x}^{\lambda}(v,v):=\Pi_{x}(v,v)-\lambda(x,v)V(\mu)>0,\quad (x,v)\in \partial_0SM.\]
Here $\mu(x,v):=\langle \nu(x),v\rangle$  and $\Pi$  is the second fundamental form of $\partial M$. The scattering relation is then a diffeomorphism of $\partial SM$, given by
\[\alpha_{\lambda}(x,v)=\phi_{\te(x,v)}(x,v),\quad (x,v)\in \partial SM.\]
Here $\te\in C^\infty(SM)$  is a function that vanishes on $\partial_0SM$ and satisfies $F\te = -2$ on $SM$ (analogous to $\tilde \tau$ in \cite[Section 3.2]{PSU23}, see also \cite{JKR24}).  The following identity is easy to check and left as an exercise for the interested reader:
\begin{equation}\label{eq:notzero}
  \Pi_{x}^{\lambda}(v,v)V(\te)(x,v)=\pm 2,\quad (x,v)\in \partial_0SM
\end{equation}
In particular, $V\te\neq 0$ on $\partial_0SM$.

\begin{lemma}\label{lemma:lambda} Suppose $\alpha_{\lambda}=\alpha_{0}$ on $\partial SM$. Then \[
\lambda|_{\partial_0SM} \equiv 0 \quad \text{ and } \quad V\lambda|_{\partial_0SM} \equiv 0.\]  In particular, if $\lambda(x,v)=\theta_x(v)$ for a $1$-form $\theta$ on $M$, then  $\theta|_{\partial M}\equiv 0$.
\end{lemma}

\begin{proof}
Since $\alpha_{\lambda}=\phi_{\te}$, for all $(x,v)\in \partial SM$ it holds that
\begin{equation}\label{chainrulealpha} d\alpha_{\lambda}(\xi)=d\te(\xi) F\circ\alpha_{\lambda}+d\phi_{\te}(\xi),\quad \xi\in T_{(x,v)}\partial SM.\end{equation}
For $(x,v)\in \partial_0SM$ this implies that $\R F(x,v)$ is an eigenspace of $d\alpha_\lambda$ with eigenvalue $-1$. Further, $T(\partial_0SM)$ is eigenspace with eigenvalue $+1$, hence:
\[
	\alpha_\lambda=\alpha_0\quad \Rightarrow \quad \R F(x,v)=\R X(x,v)\quad \Rightarrow \quad \lambda(x,v)=0.
\]

To access higher order derivatives, let $\xi\in T_{(x,v)} SM$ and write $d\phi_t(\xi) = x(t) F + y(t) H + z(t) V$. For a given initial datum $\xi$ the behaviour of the coefficients is governed by the system of ODE 
\begin{align}
\dot{x} &=\lambda\,y;\label{jac1}\\
\dot{y} &=z;\label{jac2}\\
\dot{z} &=V(\lambda)\dot{y}-\kappa y,\label{jac3}
\end{align}
where $\kappa:=K_g-H\lambda+\lambda^2$ and $K_g$ is the Gaussian curvature (one can easily derive these equations using an argument analogous to that of \cite[Proposition 3.7.8]{PSU23}). We now specify to $\xi = V(x,v)$  and denote the first two coefficients by
\[
	a_\lambda(x,v,t) = \langle d\phi_t(V(x,v)),X\rangle\quad \text{ and } \quad b_\lambda(x,v,t) = \langle d\phi_t(V(x,v)),H\rangle.
\]
Then by \eqref{chainrulealpha}, we have for all $(x,v)\in \partial SM$ that
\begin{align*}
	f_\lambda(x,v):=\langle d\alpha_\lambda(V(x,v)),X\rangle  &= V\te(x,v) + a_\lambda(x,v,\te(x,v)), \\
	g_\lambda(x,v):=\langle d\alpha_\lambda(V(x,v)),H\rangle  &=b_\lambda(x,v,\te). 
\end{align*}
These must agree with the corresponding expressions for $\alpha_0$, denoted with $f_0$ and $g_0$, respectively. Applying $V$ to the identity $f_\lambda  = f_0$ yields \[
	(V^2\te - V^2\tilde \tau_0)(x,v)= (Va_\lambda)(x,v,\te(x,v)) + \dot a_\lambda(x,v,\te(x,v)) (V\te(x,v)).
\]
As $a_\lambda(x,v,0)=\dot a_\lambda(x,v,0) = 0$ for all $(x,v)\in \partial SM$, restricting to the glancing region, where $\te(x,v)=0$, we obtain \[
V^2 \te(x,v) = V^2 \tilde \tau_0(x,v), \quad (x,v)\in \partial_0SM.\]
Next, applying $V$ twice to $g_\lambda$ yields (with $(x,v)$ suppressed from the notation):
\[
	V^2g_\lambda = V^2b_\lambda(\te) + \big(2V\dot b_\lambda (\te) + V\te\cdot  \ddot b_\lambda(\te) \big) V\te  + \dot b_\lambda(\te) V^2\te.
\]
Again, as $b_\lambda(x,v,0) = 0$ and $\dot b_\lambda(x,v,0)=1$  on $\partial SM$, all but two terms disappear at the glancing  and we obtain:
\[
	V^2g_\lambda(x,v) =\big( V\te(x,v)\big)^2 \ddot b_\lambda(x,v,0)  + V^2\te(x,v),\quad (x,v)\in \partial_0SM.
\]
Since $g_\lambda = g_0$ on $\partial SM$ and $V\te = V\tilde \tau_0\neq 0$ (by virtue of \eqref{eq:notzero}) and $V^2\te = V^2 \tilde \tau_0$ on $\partial_0SM$, we deduce
\[
	\ddot b_\lambda(x,v,0) = \ddot b_0(x,v,0),\quad (x,v)\in \partial_0SM.
\]	
Finally, using \eqref{jac2} and \eqref{jac3} to compute the second derivative, this gives $V\lambda(x,v) = 0$ for $(x,v)\in \partial_0SM$, as desired.

For the assertion about $1$-forms note that for $(x,v)\in \partial_0SM$ we have $\theta_x(v) = \lambda(x,v) = 0$ and $\theta_x(v^\perp) =- V\lambda(x,v)=0$ and thus $\theta_x =0$ by linearity.
\end{proof}

\section{Three general instances of rigidity}\label{section:general}

In this section we discuss when diffeomorphisms preserving the canonical contact 1-form, geodesic equivalences and time changes extend as biholomorphisms. 

\subsection{Diffeomorphisms preserving the contact 1-form}  Let $g_k$ ($k=1,2$) be two metrics on $M$ and consider $SM_k$ as a contact manifold with its canonical contact $1$-form $\Alpha_k$.

\begin{proposition} 
 Let $\phi:SM_{1}\to SM_{2}$ be a diffeomorphism such that $\phi^*\Alpha_{2}=\Alpha_{1}$. Then $\phi$ extends to a biholomorphism  $Z_{g_{1}} \to Z_{g_{2}}$ if and only if it is the lift of an orientation preserving isometry $\varphi\colon(M,g_{1})\to (M,g_{2})$.
\end{proposition}


\begin{proof} The lift of an orientation preserving isometry always extends as a biholomorphism between twistor spaces and it preserves contact $1$-forms.

Assume now that $\phi$ extends as a biholomorphism. By Proposition \ref{prop_hmaps}, there are functions $a,b\in C^\infty(SM_2,\R)$ such that
$\phi_{*}(V_{1})=aX_{2}+bV_{2}$. Since $\phi$ preserves the contact $1$-form, $\Alpha_{2}(\phi_{*}(V_{1}))=\Alpha_{1}(V_{1})=0$.
Thus $a=0$ and $\phi$ preserves the vertical fibres. This means that we can write $\phi$ as
\[\phi(x,v)=(\varphi(x),\psi_{x}(v)),\]
where $\varphi\colon M\to M$ is a diffeomorphism and $\psi_{x}\in\text{Diff}(S_{x}M_{1},S_{\varphi(x)}M_{2})$.
We now use that $\phi^*\Alpha_{2}=\Alpha_{1}$ by writing
\begin{align*}
(\phi^*\Alpha_{2})_{(x,v)}(\xi)&=\langle d\pi_{2}\circ d\phi_{(x,v)}(\xi),\psi_{x}(v)\rangle_{2}\\
&=\langle d\varphi_{x}\circ d\pi_{1}(\xi),\psi_{x}(v)\rangle_{2}\\
&=\langle d\pi_{1}(\xi),v\rangle_{1}=(\Alpha_{1})_{(x,v)}(\xi),
\end{align*}
where $\xi\in T_{(x,v)}SM_{1}$.
Since any $u\in T_{x}M$ can be written as $d\pi_{1}(\xi)$ for $\xi\in T_{(x,v)}SM_{1}$ we have
\[\langle d\varphi_{x}(u),\psi_{x}(v)\rangle_{2}=\langle u,v\rangle_{1} \text{ for all}\; u\in T_{x}M.\]
This yields
\[\psi_{x}(v)=((d\varphi_{x})^*)^{-1}(v).\]
Since $\psi_{x}(v)$ has norm one it follows that $d\varphi_{x}$ is an isometry and $\psi_{x}(v)=d\varphi_{x}(v)$ as desired.
Since $\phi$ extends as a biholomorphism, $\varphi$ must be orientation preserving.
\end{proof}

\subsection{Geodesic equivalences}
We say that two metrics $g_1$ and $g_2$ are {\it geodesically equivalent}, if the scaling map 
\[\phi:SM_{1}\to SM_{2}\]
given by $\phi(x,v)=(x,v/|v|_{2})$ satisfies
$\phi_{*}(X_{1})=aX_{2}$, where $a$ is a (positive) smooth function. Since $\phi$ preserves the vertical fibration
it is a candidate to extend as a biholomorphism of transport twistor space. However, we show next that if that is the case
then $g_1$ is a constant multiple of $g_2$.

\begin{proposition}\label{prop_geodeq} Suppose that $g_1$ and $g_2$ are geodesically equivalent and that the scaling map $\phi$ extends as a biholomorphism between transport twistor spaces. Then $g_1$ is a constant multiple of $g_2$.
\end{proposition}

\begin{lemma}\label{lem_locrigid}
Let $a,b,\in \C$ and define $f\colon  \mathbb S^1\to \mathbb S^1$ by $\displaystyle f(\mu)={(a\mu + b\bar \mu)}/{|a\mu + b\bar \mu|}$. Then $f$ extends to a holomorphic map $\DD\to \DD$ if and only if $b=0$.
\end{lemma}

\begin{proof}
	Assume without loss of generality that $a=1$, then for $\mu\in \mathbb S^1$ we have $f(\mu)^2  = {(\mu^2+b)}/{(\bar b \mu ^2 + 1)}$. This has a meromorphic extension to $\C\cup\{\infty\}$ with two poles of modulus $|b|^{-1/2}$ and two zeroes of modulus $|b|^{1/2}$, both of order $1$, unless $b=0$. If $f$ extends holomorphically to $\DD$, the poles must lie outside of $\DD$, which enforces $|b|<1$. Further, the zeroes of the square of a holomorphic function must be of order $2$, hence $b=0$.
\end{proof}

\begin{proof}[Proof of Proposition \ref{prop_geodeq}] 

\noindent{\it Step 1 (fixed conformal class)}: If $g_2=e^{2\sigma}g_1$ for some $\sigma\in C^\infty(M,\R)$, then the Levi-Civita connections are related by
\[
	\nabla_\xi^{g_2}\xi - \nabla^{g_1}_\xi \xi  = - g_1(\xi,\xi)\,\mathrm{grad}_{g_1}\sigma +2 (d\sigma(\xi)) \xi,\quad \xi\in C^\infty(M,TM).
\]
Being geodesically equivalent forces  this expression to be a multiple of $\xi$, which implies $d\sigma(\xi_\perp)=0$ (see \cite{Met21} for details). Hence $\sigma$ is constant, as desired.

\smallskip
\noindent{\it Step 2 (fixed fibre)}: Assume $\Phi$ is a holomorphic extension of $\phi$ that sends fibres into fibres.
 For fixed $x\in M$, we then get a holomorphic map $\Phi(x,\cdot)\colon Z_{1,x}\to Z_{2,x}$ between disks that is of the form $v\mapsto v/|v|_{g_2}$ at the boundary $\partial Z_{1,x}$. Lemma \ref{lem_locrigid} thus implies that $\Id\colon (T_xM,g_1)\to (T_xM,g_2)$ is conformal. Hence the two metrics lie in the same conformal class, and Step 1 applies.
 
 \smallskip
\noindent{\it Step 3 (general case)}: We will reduce to Step 2 by showing that every holomorphic extension preserves fibres. For this we use the {\it projective} twistor space $Z_{\mathbb{P}}$ obtained as the quotient of $Z$ by the antipodal map $a$. This is another (degenerate) complex surface such that the projection map $Z\to Z_{\mathbb{P}}$ is a 2--1 branched cover over the zero section of $Z$ (see Remark \ref{rem_proj}). 

Suppose $\phi$ extends as a biholomorphism $\Phi\colon Z_{1}\to Z_{2}$. Then $a\circ\Phi\circ a$ is a biholomorphism such that
$a\circ\Phi\circ a|_{SM_{1}}=\phi$. By the identity principle for biholomorphisms (Proposition \ref{prop_identityprinciple}) we deduce that $a\circ\Phi\circ a=\Phi$ everywhere and thus $\Phi$ descends to a biholomorphism $\tilde{\Phi}\colon Z_{\mathbb{P},1}\to Z_{\mathbb{P},2}$. 
On the other hand, it is well known that $\phi$ always induces a biholomorphism $\Phi_{\mathbb{P}}$ between the projective twistor spaces simply by extending it in the obvious way on each fibre. By the identity principle again (this time applied to projective twistor space (cf.\,\cite{BMP24}) we see that $\tilde{\Phi}=\Phi_{\mathbb{P}}$ and thus $\tilde \Phi$ sends fibres into fibres. The same property follows for $\Phi$ and we are done.
\end{proof}

\subsection{Time changes} Next we consider maps $\phi=\varphi_\tau\colon SM\to SM$ of the form
\[
	\phi(x,v) = \varphi_{\tau(x,v)}(x,v),\quad (x,v)\in SM,
\]
where $(\varphi_t)$ is the geodesic flow  and $\tau\colon SM\to \R$ is a smooth map satisfying $X\tau + 1 >0$. If $\partial M\neq 0$ we additionally assume that $\tau =0$ on $\partial SM$. We refer to $\phi$ as a {\it time change}.

\begin{lemma} The map $\phi = \varphi_\tau$ is a diffeomorphism of $SM$ (with $\phi=\mathrm{Id}$ on $\partial SM$) such that $\phi_*(qX)=X$, where $1/q=1+X\tau$.
	\label{lemma:aux}
\end{lemma}

\begin{proof}  We compute
\begin{align*}
d\phi(X(x,v))&=d\varphi_{\tau(x,v)}(X(x,v))+d\tau(X(x,v))X(\phi(x,v))\\
&=X (\phi (x,v))+d\tau(X(x,v))X(\phi(x,v))\\
&=(1+d\tau(X(x,v)))X(\phi(x,v)),
\end{align*}
which gives the desired formula for $q$. That $\phi$ is indeed a diffeomorphism can be proved by standard arguments that we omit.
\end{proof}


\begin{proposition}\label{timerigidity}
Suppose the time change $\phi=\varphi_\tau\colon SM\to SM$ extends to an automorphism of transport twistor space. Assume either of the following:
\begin{enumerate}[label=\rm(\roman*)]
	\item $\partial M= \emptyset$ and $M$ is complete and free of conjugate points
	\item $\partial M \neq \emptyset$ and $M$ is compact (and $\tau=0$ on $\partial SM$)
\end{enumerate}
Then $\tau \equiv 0$ and thus $\phi\equiv \mathrm{Id}$.
\end{proposition}

\begin{proof}
	Assume first that $M$ is complete and without boundary. Analogous to the proof of Lemma \ref{lemma:aux} one computes 
	\[
		d\phi(V(x,v)) = d\varphi_{\tau(x,v)}(V(x,v)) + V\tau(x,v) X(\phi(x,v)).
	\]
	Using Proposition \ref{prop_hmaps} we obtain $\langle d\varphi_{\tau(x,v)}(V(x,v)),H(x,v)\rangle=0$ and thus there is a dichotomy: either $\tau(x,v)=0$ or the pair of points $x$ and $y = \pi \circ \varphi_{\tau(x,v)}(x,v)$ are conjugate to each other.  If $M$ is free of conjugate points, only the first option is available and thus $\tau =0$ everywhere.
	
	Next assume that $M$ is compact with non-empty boundary. We may embed this isometrically into a closed surface $N$. By the dichotomy above, $\tau(x,v) = 0$ or $\tau(x,v)\ge \mathrm{inj}(N)$ (the injectivity radius) for any $(x,v)\in SM$. Hence $SM=\{\tau <\mathrm{inj}(N)/2\}\cup \{\tau >\mathrm{inj}(N)/2\}$	 is an open cover and by connectedness only one of the inequalities occurs. Since $\tau|_{\partial SM} =0$, we have $\tau= 0$ everywhere.
\end{proof}

\section{Biholomorphism rigidity for simple surfaces}\label{section:mainthm}

In this section we prove Theorem \ref{mainthm}. We assume throughout that $g_1$ and $g_2$ are simple metrics on $M$ with $\mathrm{Vol}(\partial M,g_1)=\mathrm{Vol}(\partial M,g_2).$

 Since the boundary lengths match up, there is an isometry $(\partial M,g_1)\to (\partial M,g_2)$ and upon extending this to a diffeomorphism of $M$, we may assume that $g_1=g_2$ on $T(\partial M)$. Further putting the metrics into a {\it normal gauge} (cf.~\cite[Proposition 11.2.1]{PSU23}), it can be arranged that
 \[
 	g_1 = g_2 \text{ on } TM|_{\partial M}
 \]
 and we will make this assumption from now on. As a consequence, both the normal vector $\nu$ and the boundary $\partial SM$ are independent of the metric.

\subsection{Boundary action} 
By the following proposition, biholomorphisms are determined by their action on $\partial SM$.

\begin{proposition}
If $\Phi,\Phi'\colon Z_{1}\to Z_{2}$ are two biholomorphisms with $\Phi=\Phi'$ on $\partial SM$, then $\Phi=\Phi'$ everywhere.
\label{prop_identityprinciple2}
\end{proposition}

\begin{proof}
		Assume first that $g_1=g_2$ and drop subscripts. Consider $\Psi = (\Phi')^{-1}\circ \Phi \in \mathrm{Aut}(Z)$, which satisfies $\Psi=\Id$ on $\partial SM$. By Proposition \ref{prop_hmaps}, the restriction $\psi = \Psi|_{SM}$ satisfies $\psi_*(q X)=X$ for a smooth function $q\colon SM\to (0,\infty)$, which implies that $\psi(x,v) = \varphi_{\tau(x,v)}(x,v)$ for a function $\tau\colon SM\to \R$ with $\tau|_{\partial SM} = 0$. One readily checks that along each orbit of the geodesic flow $\tau$ is smooth and satisfies $X\tau = 1/q-1$, such that \cite[Theorem 5.3.6]{PSU23} implies  $\tau\in C^\infty(SM)$.
		
		Hence $\psi = \varphi_\tau$ is a time change and since it extends to an automorphism,  Proposition \ref{timerigidity} implies that $\psi=\Id$. By the identity principle (Proposition \ref{prop_identityprinciple}), we must have $\Psi=\Id$ and thus $\Phi = \Phi'$, as desired. 
		
		The proof carries over to distinct metrics, with $SM$ replaced by $SM_1$.
\end{proof}

The next result describes this boundary action. Recall that $\alpha_k\in \Diff(\partial SM)$  is the {\it scattering relation} of $g_k$  $(k=1,2)$.

\begin{proposition}
\label{boundaryaction} 
If $\Phi\colon Z_1\to Z_2$ is a biholomorphism, then $\phi = \Phi|_{\partial SM}\in \mathrm{Diff}(\partial SM)$ satisfies the following properties:
\begin{enumerate}[label=\rm(\roman*)]
	\item \label{boundaryaction1} $\alpha_2\circ \phi = \phi \circ \alpha_1$;
	\item \label{boundaryaction2}$\phi_*V = p V$ for some $p\in C^\infty(\partial SM)$;
	\item \label{boundaryaction3}if $u\in C^\infty(\partial SM)$ is fibrewise holomorphic and $\alpha_2$-invariant, then also $\phi^*u$ is fibrewise holomorphic (and by \ref{boundaryaction1} automatically $\alpha_1$-invariant).
\end{enumerate}
\end{proposition}

\begin{proof}
Property \ref{boundaryaction1} follows immediately from Proposition \ref{prop_hmaps}\ref{hmaps1}. This also yields $\phi_*(V) = q X + p V$ for smooth functions $p,q\in C^\infty(\partial SM)$ and since $X$ is transversal to $\partial SM\backslash \partial_0SM$ (which is a dense subset), we must have $q\equiv 0$ and obtain property \ref{boundaryaction2}. Since $\Phi^*\colon \mathcal A(Z_2)\to \mathcal A(Z_1)$, the third property is an immediate consequence of Proposition \ref{prop_scathol}.
\end{proof}

\subsection{Rigidity of the conformal class}

\begin{proposition}\label{rigidity1}
	 Suppose  $\phi\in \mathrm{Diff}(\partial SM)$ extends to a biholomorphism $Z_1\to Z_2$ (or, just satisfies properties \ref{boundaryaction1},\ref{boundaryaction2} and \ref{boundaryaction3} of Proposition \ref{boundaryaction}). Then there exists a biholomorphism $\varphi\colon (M,g_1)\to (M,g_2)$ such that
	\[
		\phi(x,\mu\cdot \nu_\perp(x)) = (\varphi(x), \mu\cdot \nu_\perp(\varphi(x)),\quad x\in \partial M,\mu\in \mathbb S^1,
	\]
	possibly after composing $\phi$ with the antipodal map. Moreover, writing $\varphi^*g_2=e^{2\sigma}g_1$ for some $\sigma\in C^\infty(M,\R)$, the function $e^\sigma - 1$ has zero average on $\partial M$. 
\end{proposition}

\begin{remark}
	One can further show that $\sigma|_{\partial M} = \log(\pi_1/\varphi^*\pi_2)$, where $\pi_k(x) = \Pi_{k,x}(\nu_\perp(x),\nu_\perp(x))$ and $\Pi_{k}$ is the second fundamental form of $\partial M$.  If $\pi_1$ and $\pi_2$ are both constant, it then follows immediately that $\sigma|_{\partial M} = 0$. In Lemma \ref{lem_boundarydetsigma} below this condition on the second fundamental forms is removed.
\end{remark}

As preparation we need a lemma. Recall from Section \ref{sec_RKC} that $\mathbb H\subset L^2(\mathbb S^1)$ is the Hardy space and, for $x\in \partial M$, denote with $\iota_x\colon \mathbb S^1\rightarrow Z$ the map $\iota_x(\mu)= \mu \cdot \nu_\perp(x)$.

\begin{lemma}\label{lem_pu} If $(M,g)$ is simple and $x\in \partial M$, then the range of $\iota_x^*\colon \mathcal A(Z)\rightarrow \mathbb H$ is dense in the weak topology.
\end{lemma}

\begin{proof}
We construct an element $A= \iota_x^*(\mathcal A(Z)) \in \mathbb H$ such that
\[
	A(\mu) = \mu  + A_3 \mu^3 + \dots
\]
Let $a\colon M\to \C$ be holomorphic with $da_x(\nu_\perp(x))=1$. Then $da \in \mathcal H_1(M)$ and by Proposition \ref{puext2} there is function  $\hat a \in \mathcal A(Z)$ with $\hat a (x,\mu\cdot\nu_\perp(x)) = \mu + O(|\mu|^3)$. By the Cauchy integral formula,  $A=\iota_x^*\hat a$ has the desired Fourier modes.

It is easy to check that the linear span of $\{1,A,A^2,A^3,\dots\}\subset \mathbb H$ is weakly dense and from this the lemma follows.
\end{proof}

\begin{proof}[Proof of Proposition \ref{rigidity1}]
By  property \ref{boundaryaction2} the diffeomorphism $\phi$ is of the form $\phi(x,v)=(\varphi(x),\psi_x(v))$, where $\varphi\in \Diff(\partial M)$ and $\psi_x\in \Diff(S_xM ,S_{\varphi(x)}M)$. Due to \ref{boundaryaction1}, $\phi$ must map the glancing region $\partial_0SM$ into itself and thus  either
\[
	\psi_x(\nu_\perp(x)) = \nu_\perp(\varphi(x))\quad \text{ or } \quad \psi_x(\nu_\perp(x)) = - \nu_\perp(\varphi(x)),\qquad x\in \partial M,
\]
with a consistent choice of sign due to continuity. After applying the antipodal map if necessary, we may assume that we are in the first case.

\smallskip
\noindent {\it Step 1 (Oddness)} We claim that $\psi_x$ is odd in the sense that $\psi_x(-v)=-\psi_x(v)$ for all $(x,v)\in \partial SM$. To this end, put $(y,w)=\alpha_1(x,v)$ and  $u_\pm = \psi_x(\pm v)\in S_{\varphi (x)} M$.  We have to show that $u_+=-u_-$. Applying \ref{boundaryaction1} to $(x,\pm v)$ yields:
 \begin{equation*}
 \alpha_2(\varphi(x),u_\pm) = (\varphi(y),\psi_y(\pm w))\quad \Rightarrow \quad \pi\circ \alpha_2(\varphi(x),u_+)=\pi\circ \alpha_2(\varphi(x),u_-)
 \end{equation*}
Hence the $g_2$-geodesics through $(\varphi(x),u_+)$ and $(\varphi(x),u_-)$ have the same starting and end point and due to simplicity of $g_2$ they must agree, which enforces $u_+=u_-$ or $u_+=-u_-$. The first case is ruled out by the injectivity of $\psi_x$.

\smallskip
\noindent {\it Step 2 (Holomorphicity).} Let $x\in \partial M$ and define $\psi\in \Diff(\mathbb S^1)$ such that
\[
	\phi\circ \iota_x  = \iota_{\varphi(x)}\circ \psi\quad \text{ on } \mathbb S^1
\]
We claim that $\psi$ is holomorphic in the sense that $\psi^*\mathbb H\subset \mathbb H$. Suppose that $h\in \mathbb H$ is given as $h = f\circ \iota_{\varphi(x)}$ for some $f\in \mathcal A(Z_2)$, then due to \ref{boundaryaction3} we have $\psi^* h = \iota_x^* \phi^* f \in \mathbb H$. Now for general $h\in \mathbb H$ we can apply Lemma \ref{lem_pu} to obtain a sequence $h^{(N)}=f^{(N)}\circ \iota_{\varphi(x)}$ with $f^{(N)}\in \mathcal A(Z_2)$ and $h^{(N)}\rightharpoonup h$. By the preceding considerations, $\psi^* h^{(N)}\in \mathbb H$ and since the pull-back $\psi^*\colon  \mathbb H\rightarrow L^2(\mathbb S^1)$ is weak-weak 
continuous we must have $\psi^*h^{(N)}\rightharpoonup \psi^* h$ in $L^2(\mathbb S^1)$. Finally,  every norm-closed subspace of $L^2(\mathbb S^1)$ is also weakly closed and hence $\psi^*h\in \mathbb H$.

\smallskip
\noindent {\it Step 3 (Fibrewise rigidity).} We keep focusing on the diffeomorphism $\psi\in \Diff(\mathbb S^1)$ for a fixed $x\in \partial M$. By construction, $\psi_x(\mu \cdot \nu_\perp(x)) = \psi(\mu) \cdot \nu_\perp(\varphi(x))$ and $\psi(1)=1$. Further, oddness of $\psi_x$ implies oddness of $\psi$ and since we have just shown that $\psi$ is also holomorphic, we must have $\psi\equiv \Id$ by Lemma \ref{circlediff}.

\smallskip
\noindent {\it Step 4 (Extension).} It suffices to show that $\varphi$ satisfies the assumption of Lemma \ref{rkcriemann}. To this end, let $h\in \mathcal A(\partial M,g_2)$. We can first extend this $g_2$-holomorphically to $M$ and then by Proposition \ref{puext1} to a function $u\in \mathcal A(Z_2)$. 
Then $u|_{\partial SM}$ is $\alpha_2$-invariant and fibrewise holomorphic and property  \ref{boundaryaction3} ensures that also $\phi^*(u|_{\partial SM})\in C^\infty(\partial SM)$ is fibrewise holomorphic (and $\alpha_1$-invariant). Hence it can be extended to a function $w\in \mathcal A(Z_1)$. By construction, if $x\in \partial M,\mu\in \mathbb S^1$,
\[
	h(\varphi(x)) = u(\varphi(x),0) \quad \text{ and } \quad u(\varphi(x),\mu\cdot \nu_\perp(\varphi(x))) = w(x,\mu\cdot \nu_\perp(x))
\]
The second equation can be analytically continued to $|\mu|\le 1$ and plugging in $\mu=0$ yields $\varphi^*h = w(\cdot,0)|_{\partial M}\in \mathcal A(\partial M,g_1)$, as desired.

\smallskip
\noindent {\it Step 5 (Boundary behaviour).}
The boundary lengths of $g_1$ and $e^{2\sigma} g_1$ agree, and since the length element of the latter metric changes by a factor of $e^\sigma$ on $\partial M$, it follows that $1-e^\sigma$ has zero average with respect to $g_1$.
\end{proof}

\subsection{Rigidity within a conformal class}
The problem is now reduced to determining a conformal factor. Indeed, starting with two simple metrics $g_1$ and $g_2$ with $g_1= g_2$ on $TM|_{\partial M}$ and a biholomorphism  $\Phi\colon Z_1\to Z_2$ between their twistor spaces, Propositions \ref{boundaryaction} and \ref{rigidity1} guarantee the existence of an orientation preserving diffeomorphism $\varphi\in \Diff (M)$ and $\sigma\in C^\infty(\DD)$ such that (possibly after composing with the antipodal map),
\[
	\varphi^*g_2= e^{2\sigma} g_1\quad \text{ and } \quad \Phi|_{\partial SM} = \varphi_\sharp \circ \mathrm{sc}_\sigma|_{\partial SM}.
\]
Here $\mathrm{sc}_\sigma\colon SM\to SM_\sigma:=\{(x,w)\in TM: |w|_{e^{2\sigma}}=1\} $, $\mathrm{sc}_\sigma(x,v)=(x,e^{-\sigma(x)}v)$ is the rescaling map. Hence, writing $g=g_1$, the composition $\Psi:=\varphi_{\sharp}^{-1}\circ \Phi$ defines  a biholomorphism
\begin{equation}\label{conformalbiholo}
	\Psi\colon Z_g\to Z_{e^{2\sigma}g},\qquad \Psi|_{\partial SM} = \mathrm{sc}_\sigma|_{\partial SM}.
\end{equation}
Our task is to prove that $\sigma\equiv 0$ on $M$, in which case Proposition \ref{prop_identityprinciple2} implies that $\Psi \equiv \mathrm{Id}$ and the proof of Theorem \ref{mainthm} is complete. 

To proceed, we translate the task into a {\it scattering rigidity} problem. Indeed, \eqref{conformalbiholo} implies the following relation between scattering relations:
\begin{equation}\label{scatteringalmostequivalence}
	\alpha_{e^{2\sigma}g}\circ \mathrm{sc}_\sigma = \mathrm{sc}_\sigma
	\circ\alpha_g\quad \text{ on } \partial SM
\end{equation}
Assuming that $\sigma|_{\partial M} =0$, Remark 11.3.5 and Theorem 11.4.1 in \cite{PSU23} imply that $\sigma \equiv 0$, as desired. Given that $e^\sigma-1$ has zero average, it remains to establish the following boundary determination result.

\begin{lemma}\label{lem_boundarydetsigma}
Assume $(M,g)$ and $(M,e^{2\sigma}g)$ are both non-trapping with strictly convex boundary.  Then \eqref{scatteringalmostequivalence} implies that $\sigma|_{\partial M}$ is constant.
\end{lemma}


\begin{proof}[Proof of Lemma \ref{lem_boundarydetsigma}]
We introduce a thermostat flow on $SM=SM_g$, generated by the vector field
\[
	F  = X + \lambda V,\quad \lambda(x,v) = -\star d\sigma_x(v).
\]
On account of \cite[Lemma B.1]{CP22}, we have
$
	(\mathrm{sc}_\sigma)_*(F) = e^{-\sigma} X_{e^{\sigma}g}.
$	 
and thus $\mathrm{sc}_\sigma^{-1}\circ \alpha_{e^{2\sigma}g}\circ \mathrm{sc}_\sigma$ equals the scattering relation $\alpha_\lambda$ of the thermostat $F$ (see Section \ref{seclambda}). 
By Lemma \ref{lemma:lambda} we must have $d\sigma|_{\partial M} =0$, as desired.
\end{proof}

\section{Biholomorphism rigidity for Anosov surfaces}\label{section:mainthm_anosov}

In this section we prove Theorem \ref{mainthm_anosov}. We start with a proposition that is probably well known to experts, but we include it for completeness.

\begin{proposition} Let $g_{1}$ and $g_{2}$ be two Anosov metrics on a closed surface $M$ with the same area and such that there exists
a smooth diffeomorphism $\phi:SM_{1}\to SM_{2}$ with $\phi_{*}(qX_{1})=X_{2}$ for a {\it positive} smooth function $q$. 
Then
\[1/q=1+X_{1}u\]
for some $u\in C^{\infty}(SM_{1})$.
\label{prop:aux}

\end{proposition}

\begin{proof}

Let $\boldsymbol{\alpha}_{i}$ denote the canonical contact form of $SM_{i}$ for $i=1,2$.
We know that $\phi$ is a smooth orbit equivalence between geodesic flows and by transitivity there is a constant $c$ such that
$\phi^*d\boldsymbol{\alpha}_{2}=c\,d\boldsymbol{\alpha}_{1}.$
Hence, there exists a closed 1-form $\omega$ on $SM_{1}$ such that
\begin{equation}
\phi^*\boldsymbol{\alpha}_{2}=c\,\boldsymbol{\alpha}_{1}+\omega.
\label{eq:eqc}
\end{equation}
Next we observe that 
\[\phi^*(\boldsymbol{\alpha}_{2}\wedge d\boldsymbol{\alpha}_2)=c^2\,\boldsymbol{\alpha}_{1}\wedge d\boldsymbol{\alpha}_{1}+c\,\omega\wedge d\boldsymbol{\alpha}_{1}.\]
Since $\omega$ is closed, the form $\omega\wedge d\boldsymbol{\alpha}_1$ is exact, hence Stokes' theorem gives
\[\int_{SM_{1}}\phi^*(\boldsymbol{\alpha}_{2}\wedge d\boldsymbol{\alpha}_2)=c^2\int_{SM_{1}}\boldsymbol{\alpha}_{1}\wedge d\boldsymbol{\alpha}_1.\]
Since $\phi$ is a diffeomorphism and the surfaces have the same area, this implies $c^{2}=1$. We claim that $c=1$.
Contracting 
\eqref{eq:eqc} with $qX_{1}$ gives
\[1=q(c+\omega(X_{1}))\]
and integrating over $SM_{1}$ against any invariant probability measure $\mu$ with zero homology (of which there are many, like the Liouville measure, see \cite[Section 2.6]{Pat99}) one gets
\[c=\int_{SM_{1}}q^{-1}d\mu>0.\]
It follows that 
$1=q(1+\omega(X_{1}))$.
To complete the proof of the proposition we will show that $\omega$ is exact and writing $\omega=du$ we obtain the desired result.

To kill the cohomology class $[\omega]$ we will use the well known fact that the measure of maximal entropy of the geodesic flow has zero homology/winding cycle (which is simply because the antipodal map conjugates the geodesic flow $\varphi_{t}^{1}$ with $\varphi_{-t}^{1}$, cf. \cite[Section 2.6]{Pat99}). Consider the pressure function in cohomology $\beta_{X_{1}}:H^{1}(SM_{1},\mathbb{R})\to\mathbb{R}$, given by
\[\beta_{X_{1}}([\omega])=\sup_{\mu}\left\{h_{\mu}+\int_{SM_{1}}\omega(X_{1})d\mu\right\},\]
where $\mu$ runs over all invariant Borel probability measures and $h_{\mu}$ denotes the metric entropy of $\mu$.
Since the measure of maximal entropy has zero winding cycle, we see that $\beta_{X_{1}}([\omega])\geq h_{\text{top}}(X_{1})=\beta_{X_{1}}(0)$.
The pressure function is strictly convex, therefore it has a unique minimiser $\xi_{X_{1}}$ called the {\it Sharp minimiser}, see \cite{Sh93}. Since $\beta_{X_{1}}(0)$ is the absolute minimum, the Sharp minimiser for the geodesic flow must be zero.

Finally, since there is a time preserving conjugacy between $Z:=X_{1}/(1+\omega(X_{1}))$ and $X_{2}$, the Sharp minimiser $\xi_{Z}$ of $Z$ is also zero and hence $\omega$ must be exact. Indeed,
\cite[Proposition 4.4]{GR-H_24} (the geodesic flow is homologically full) gives
\[\xi_{Z}=\xi_{X_{1}}+P_{X_{1}}(\xi_{X_{1}}(X_{1}))[\omega]\]
and $P_{X_{1}}(0)=h_{\text{top}}(X_{1})\neq 0$ (here $P_{X_{1}}(f)$ is the pressure of the potential $f$).
\end{proof}

\begin{proof}[Proof of Theorem \ref{mainthm_anosov}]

    By Proposition \ref{prop_hmaps}, the biholomorphism $\Phi$ induces a diffeomorphism $\phi\colon SM_{1}\to SM_{2}$ such that $\phi_{*}(qX_{1})=X_{2}$ for some smooth, non-vanishing function $q\colon SM_1 \to \mathbb{R}$, 
    and after correcting with the antipodal map if necessary, we may assume that $q$ is positive on $SM_1$. 


    Proposition \ref{prop:aux} and Lemma \ref{lemma:aux} imply that $X_{1}$ and $X_{2}$ are smoothly conjugate via the map $\phi\circ(\varphi^{1}_{u})^{-1}$. Now we can apply \cite[Corollary 1.3]{GLP23} to obtain an isometry $F\colon (M,g_{1})\to (M,g_{2})$ such that $\phi=F_{\sharp}\circ \varphi^{1}_{\tau}$, where $\tau \in C^{\infty}(SM_{1})$. But then Proposition \ref{timerigidity} allows us to conclude that $\tau =0$ and by the identity principle (Proposition \ref{prop_identityprinciple}), $\Phi=F_{\sharp}$ in all $Z_{1}$, as desired.
\end{proof}

\bibliographystyle{plain}
\bibliography{Bihol_paper.bbl}

\end{document}